\documentclass[12pt]{amsart}
\usepackage{amssymb}
\usepackage{mathrsfs}
\renewcommand\b{\beta}

\renewcommand\d{\delta}
\newcommand\la{\lambda}

\newcommand\e{\eta}

\newcommand\io{\iota}

\newcommand\f{\phi}
\newcommand\vf{\varphi}
\newcommand\p{\psi}

\newcommand\w{\omega}

\newcommand\vS{\varSigma}
\newcommand\D{\Delta}

\newcommand\F{\Phi}

\newcommand\vL{\varLambda}

\newcommand\ve{\varepsilon}

\newcommand{\ZZ}{\mathbb Z}

\newcommand\BP{\mathbf P}

\newcommand\Ba{\mathbf a}

\newcommand\Bp{\mathbf p}
\newcommand\Bm{\mathbf m}
\newcommand\Bb{\mathbf b}

\newcommand\ZC{\mathcal{C}}

\newcommand\CP{\mathcal{P}}

\newcommand\CT{ \mathcal{T}}

\newcommand\wh{\widehat}
\newcommand\wt{\widetilde}

\newcommand\ol{\overline}

\newcommand\trreq{\trianglerighteq}
\newcommand\hra{\hookrightarrow}

\newcommand\LRa{\Leftrightarrow }

\newcommand{\lan}{\langle}
\newcommand{\ran}{\rangle}

\newcommand{\ap}{\Ba_\Bp}
\newcommand{\trr}{\triangleright }

\newcommand{\ra}{\rightarrow }

\newcommand\Ker{\operatorname{Ker}}
\newcommand\Hom{\operatorname{Hom}}

\newcommand\id{\operatorname{id}}

\newcommand\Std{\operatorname{Std}}
\newcommand\GL{\operatorname{GL}}
\newcommand{\diag}{\operatorname{diag}}
\newcommand{\rad}{\operatorname{rad}}

\newcommand{\isom}{\,\raise2pt\hbox{$\underrightarrow{\sim}$}\,}
\newcommand{\Sc}{\mathscr{S}}
\newcommand{\Sp}{\mathscr{S}^\Bp}
\newcommand{\oSp}{\ol{\Sc}^\Bp}
\newcommand{\alp}{\alpha_\mathbf{p}}
\newcommand{\He}{\mathscr{H}}
\newcommand{\fs}{\mathfrak{s}}
\newcommand{\ft}{\mathfrak{t}}
\newcommand{\fsl}{\mathfrak{sl}}
\newcommand{\fu}{\mathfrak{u}}
\newcommand{\fv}{\mathfrak{v}}

\newtheorem{thm}{Theorem}[section]
\newtheorem{lem}[thm]{Lemma}
\newtheorem{cor}[thm]{Corollary}
\newtheorem{prop}[thm]{Proposition}

\def \para{\refstepcounter{thm} \par\medskip\noindent
                \textbf{\thethm .} }

\def \remark{\refstepcounter{thm} \par\medskip\noindent
                \textbf{Remark \thethm .} }

\newcounter{ichi}
\newcommand{\roi}{\roman{ichi}}
\newcounter{ni}
\newcommand{\roii}{\roman{ni}}
\newcounter{san}
\newcommand{\roiii}{\roman{san}}

\numberwithin{equation}{thm}

\begin{document}
\setlength{\baselineskip}{4.9mm}
\setlength{\abovedisplayskip}{4.5mm}
\setlength{\belowdisplayskip}{4.5mm}
\renewcommand{\theenumi}{\roman{enumi}}
\renewcommand{\labelenumi}{(\theenumi)}
\renewcommand{\thefootnote}{\fnsymbol{footnote}}
\renewcommand{\thefootnote}{\fnsymbol{footnote}}
\parindent=20pt
\medskip
\begin{center}
{\bf On decomposition numbers with Jantzen filtration \\of cyclotomic $q$-Schur algebras} 
\\
\vspace{1cm}
Kentaro Wada 
\\ 
\vspace{0.5cm}
Graduate School of Mathematics \\
Nagoya University  \\
Chikusa-ku, Nagoya 464-8602,  Japan
\end{center}
\title{}
\maketitle
\markboth{Kentaro Wada}{On decomposition numbers with Jantzen filtration}
\begin{abstract}
Let $\Sc(\vL)$ be the cyclotomic $q$-Schur algebra associated to the Ariki-Koike algebra $\He_{n,r}$, 
introduced by Dipper-James-Mathas. In this paper, we consider  $v$-decomposition numbers of $\Sc(\vL)$, 
namely decomposition numbers with respect to the Jantzen filtrations of Weyl modules. 
We prove, as a $v$-analogue of the result obtained by Shoji-Wada, a product formula for $v$-decomposition numbers of $\Sc(\vL)$, which asserts that 
certain $v$-decomposition numbers are expressed as a product of $v$-decomposition numbers for various 
cyclotomic $q$-Schur algebras associated to Ariki-koike algebras $\He_{n_i,r_i}$ of smaller rank. 
Moreover we prove a similar formula for $v$-decomposition  numbers of $\He_{n,r}$ by using a Schur functor.
\end{abstract}
\setcounter{section}{-1}
\section{Introduction}
Let $\He=\He_{n,r}$ be the Ariki-Koike algebra over an integral domain $R$ 
associated to the complex reflection group $\mathfrak{S}_n \ltimes (\ZZ/r\ZZ)^n$. 
Dipper, James and Mathas \cite{DJM98} introduced the cyclotomic $q$-Schur algebra $\Sc(\vL)$ associated to the Ariki-Koike algebra $\He$, 
and they showed that $\He$ and $\Sc(\vL)$ are cellular algebras in the sence of Graham and Lehrer \cite{GL96},  
by constructing the cellular basis respectively. 
It is a fundamental problem for the representation theory to determine the decomposition numbers of $\He$ and $\Sc(\vL)$. 
It is well-known that the decomposition matrix of $\He$ coincides with 
the submatrix of that of $\Sc(\vL)$ by the Schur functor.

In the case where $\He$  is the Iwahori-Hecke algebra $\He_n$ of type A, 
 Lascoux, Leclerc and  Thibon \cite{LLT96} conjectured that the decomposition numbers of $\He_n$ can be described 
by using the canonical basis of a certain irreducible $U_v(\wh{\mathfrak{sl}}_e)$-module,  
and gave the algorithm to compute this canonical basis. 
The cojecture has been solved by Ariki \cite{Ari96}, by extending to the case of Ariki-Koike algebras.

In the case of the $q$-Schur algebra associated to $\He_n$, 
 Leclerc and  Thibon \cite{LT96} conjectured that the decomposition matrix coincides with the transition matrix  
between the canonical basis and the standard basis of the Fock space of level 1 equipped with the $U_v(\wh{\fsl}_e)$-module structure,  
and gave the algorithm to compute the transition matrix. 
This conjecture has been solved by Varagnolo and  Vasserot in \cite{VV99}. 

More generally, in the case of the cyclotomic $q$-Schur algebra $\Sc$, Yvonne \cite{Yvo05} has conjectured that the decomposition matrix 
coincides with the transition matrix  between the canonical basis and the standard basis 
of the higher-level Fock space. 
This canonical basis was constructed by Uglov \cite{Ugl00} and the algorithm to compute the transition matrix was also given there. 
Yvonne's conjecture is still open. 
We remark that Ariki's theorem, Varagnolo-Vasserot's theorem and Yvonne's conjecture are concerned with the situation where $R$ is a complex number field 
and parameters are roots of unity. 

In order to study the decomposition numbers of $\Sc$, we constructed in \cite{SW} some subalgebras $\Sp$ of 
$\Sc(\vL)$ and their quotients $\oSp$, and showed that $\Sp$ 
is a standardly based algebra in the sence of Du and Rui \cite{DR98}, and that $\oSp$ is a cellular algebra. 
Hence, one can consider the decomposition numbers of $\Sp$ and $\oSp$ also. 
We denote the decomposition numbers of $\Sc$, $\Sp$ and $\oSp$ 
by $d_{\la\mu}$, $d_{\la\mu}^{(\la,0)}$ and $\ol{d}_{\la\mu}$ respectively, 
where $d_{\la\mu}$ is a decomposition number of the irreducible module $L^\mu$ in the Weyl module $W^\la$ of $\Sc$ for $r$-partitions $\la,\mu$, 
and $d_{\la \mu}^{(\la,0)}$, $\ol{d}_{\la\mu}$ are defined similarly for $\Sp$ and $\oSp$ (see Section 1 for details).  
It is proved in \cite[Theorem 3.13]{SW} that \medskip \\ 
(1) \quad$\ol{d}_{\la\mu}=d_{\la\mu}^{(\la,0)}= d_{\la\mu}$ \medskip \\
whenever $\la, \mu$ satisfy a certain condition $\alp(\la)=\alp(\mu)$.  
Moreover for such $\la,\mu$, 
the product formula for  $\ol{d}_{\la\mu}$,  \medskip\\
(2) \quad$\ol{d}_{\la\mu}=\prod_{k=1}^{g} d_{\la^{[k]}\mu^{[k]}}$, \medskip\\ 
was proved in \cite[Theorem 4.17]{SW}, 
where $d_{\la^{[k]}\mu^{[k]}}$ for $k=1,\cdots ,g$ is the decomposition number of the cyclotomic $q$-Schur algebra 
associated to a certain Ariki-Koike algebra $\He_{n_k,r_k}$. 

Related to the above conjectures on Fock spaces, Leclerc-Thibon and Yvonne give a more precise conjecture concerning the $v$-decomposition numbers 
defined by using  Jantzen filtrations of Weyl modules. (For definition of $v$-decomposition numbers, see \S\ref{def-vdecom}.) 
We remark that decomposition numbers coincide with $v$-decomposition numbers at $v=1$. 
Thus we regard $v$-decomposition numbers as a $v$-analogue of decomposition numbers. 
The conjecture for $v$-decomposition numbers has been still open even in the case of the $q$-Schur algebra of type A. 

In this paper, we show that similar formula as (1) and (2) also hold for $v$-decomposition numbers. 
We denote the $v$-decomposition numbers of $\Sc(\vL)$, $\Sp(\vL)$ and $\oSp(\vL)$ 
by $d_{\la\mu}(v)$, $d_{\la\mu}^{(\la,0)}(v)$ and $\ol{d}_{\la\mu}(v)$ respectively. 
Then for $r$-partitions $\la,\mu$ such that $\alp(\la)=\alp(\mu)$, we have (Theorem \ref{th-vdecom})
\[\ol{d}_{\la\mu}(v)=d_{\la\mu}^{(\la,0)}(v)=d_{\la\mu}(v),\] 
and (Theorem \ref{th-decom-bar-v})
\[d_{\la\mu}(v)=\ol{d}_{\la\mu}(v)=\prod_{k=1}^g d_{\la^{[k]}\mu^{[k]}}(v),\]
where $d_{\la^{[k]}\mu^{[k]}}(v)$ is the $v$-decomposition number of the cyclotomic $q$-Schur algebra
appeared in (2).

We note that our result is a $v$-analogue of (1),(2), and it reduces to them by taking $v \mapsto 1$. 
Moreover, for a certain $v$-decomposition number $d_{\la\mu}^\He(v)$ of the Ariki-Koike algebra, we also have the following product formula 
(Theorem \ref{th-vdecom-AK}). 
\[d_{\la\mu}^\He(v)=\prod_{k=1}^g d_{\la^{[k]}\mu^{[k]}}^\He(v),\]
where $d_{\la^{[k]}\mu^{[k]}}^\He(v)$ is the $v$-decomposition number of the certain Ariki-Koike algebra $\He_{n_k,r_k}$.

We remark that our results hold for any parameters and any modular system, even for the case where the base field has non-zero characteristic, 
though Yvonne's conjecture is formulated under certain restrictions for parameters and modular systems.

\vspace{1em}
\noindent\textbf{Acknowledgments} I would like to thank Toshiaki Shoji and Hyohe Miyachi for many helpful advices and discussions.
\section{A review of known results}
\label{pre}
\para 
Througout the paper, we follow the notation in \cite{SW}. Here we review some of them. 
We fix positive integers $r$, $n$ and an $r$-tuple $\Bm=(m_1,\cdots ,m_r)\in \ZZ_{>0}^r$. 
A composition $\la=(\la_1,\la_2,\cdots)$ is a finite sequence of non-negative integers, 
and $|\la|=\sum_{i} \la_i$ is called the size of $\la$. If $\la_l\not=0$ and $\la_k=0$ for any $k>l$, then $l$ is called the length of $\la$. 
If the composition $\la$ is a weakly decreasing sequence, $\la$ is called a partition. 
An $r$-tuple $\mu=(\mu^{(1)},\cdots ,\mu^{(r)})$ of compositions is called the $r$-composition, 
and size $|\mu|$ of $\mu$ is  defined by $\sum_{i=1}^r|\mu^{(i)}|$. 
In particular, if all $\mu^{(i)}$ are partitions, $\mu$ is called an $r$-partition. 
We denote by $\vL=\wt{\CP}_{n,r}(\Bm)$ the set of $r$-compositions $\mu=(\mu^{(1)},\cdots ,\mu^{(r)})$ such that $|\mu|=n$ and 
that the length of $\mu^{(k)}$ is smaller than $m_k$ for $k=1,\cdots ,r$. We define $\vL^+=\CP_{n,r}(\Bm)$ as the subset of $\vL$ consisting of $r$-partitions .

We define the partial order, the so-called \lq\lq dominance order", on $\vL$ by $\mu \trreq \nu$ if and only if
\[\sum_{i=1}^{l}|\mu^{(i)}| +\sum_{j=1}^{k}\mu^{(l)}_j \geq \sum_{i=1}^{l}|\nu^{(i)}| +\sum_{j=1}^{k}\nu^{(l)}_j\]
for any $1\leq l \leq r$, $1\leq k \leq m_l$. If $\mu \trreq \nu $ and $\mu \not= \nu$, we write it as $\mu \trr \nu$. 

For $\la \in \vL^+$, we denote by $\Std(\la)$ the set of standard tableau of shape $\la$. 
For $\la \in \vL^+$ and $\mu \in \vL$, we denote by $\CT_0(\la,\mu)$ the set of semistandard $\la$-tableau of type $\mu$. 
Moreover we set $\CT_0(\la)=\cup_{\mu\in \vL}\CT_0(\la,\mu)$. 
For definitions of standard tableau and semistandard tableau, see \cite{SW} or \cite{DJM98}. 
\para 
Let $\He=\He_{n,r}$ be the Ariki-Koike algebra over an integral domain $R$ with parameters $q, Q_1,\cdots,Q_r$ with defining relations in \cite[\S 1.1]{SW}. 
It is known by \cite{DJM98} that $\He$ has a structure of the cellular algebra with a cellular basis 
$\{m_{\fs \ft}\,|\, \fs, \ft \in \Std(\la) \text{ for some }\la \in \vL^+ \}$. 
Then the general theory of a cellular algebra by \cite{GL96} implies the following results. 
There exists an anti-automorphism $h \mapsto h^\ast$ of $\He$ such that $m_{\fs \ft}^\ast=m_{\ft \fs}$. 
For $\la \in \vL^+$, let $\He^{\vee \la}$ be the $R$-submodule of $\He$ 
spaned by $m_{\fs \ft}$, where $\fs,\ft\in \Std (\mu)$ for some $\mu\in \vL^+$ such that $\mu \triangleright \la$. 
Then $\He^{\vee \la}$ is an ideal of $\He$.
 One can construct the standard (right) $\He$-module $S^\la$, called a Specht module, with 
the $R$-free basis $\{m_{\ft}\,|\, \ft\in \Std(\la)\}$. 
We define the bilinear form $\lan\,,\,\ran_{\He}$ on $S^\la$ by 
\[\lan m_{\fs}, m_{\ft} \ran_\He m_{\fu \fv}\equiv m_{\fu \fs} m_{\ft \fv} \mod \He^{\vee \la} \quad \big( \fs,\ft \in \Std({\la})\big),\]
where $\fu, \fv \in \Std(\la)$, and the scalar $\lan m_{\fs},m_{\ft}\ran_\He$ does not depend on the choice of $\fu ,\fv \in \Std(\la)$. 
The bilinear form $\lan\,,\,\ran_\He$ is associative, namely we have 
\begin{equation}
\lan xh, y \ran_\He =\lan x, yh^\ast \ran_\He \quad \text{for }x,y \in S^\la,\,\,h\in \He.
\end{equation}
Let $\rad S^\la =\{x\in S^\la\,|\, \lan x,y \ran_\He=0 \text{ for any } y\in S^\la \}$. 
Then $\rad S^\la$ is the $\He$-submodule of $S^\la$ 
by the associativity of the bilinear form. Put $D^\la = S^\la / \rad S^\la$. 
Assume that $R$ is a field. 
Then $D^\la$ is an absolutely irreducible module or zero, 
and the set $\{D^\la \,|\, \la\in \vL^+ \text{ such that } D^\la \not=0 \}$ gives a complete set of non-isomorphic irreducible $\He$-modules.
\para
Let $\Sc=\Sc(\vL)$ be the cyclotomic $q$-Schur algebra introduced by \cite{DJM98}, associated to the Ariki-Koike algebra 
$\He$ with respect to the set $\vL$. 
It is known by \cite{DJM98} that $\Sc$ is a cellular algebra with a cellular basis $\{\vf_{ST}\,|\,S,T\in\CT_0(\la) \text{ for some }\la\in \vL^+\}$. 
Again by the general theory of a cellular algebra, the following results hold. 
There exists the anti-automorphism $x \mapsto x^\ast$ of $\Sc$ such that $\vf_{ST}^\ast=\vf_{TS}$. 
For $\la \in \vL^+$, let $\Sc^{\vee \la}$ be the $R$-submodule spaned by 
$\vf_{ST}$, where $S,T\in \CT_0(\mu)$ for some $\mu \in \vL^+$ such that $\mu \trr \la$. 
Then $\Sc^{\vee \la}$ is an ideal of $\Sc$. 
One can construct the standard (right) $\Sc$-module $W^\la$ ($\la \in \vL^+$),  
called a Weyl module, with the $R$-free basis $\{\vf_{T}\,|\,T\in \CT_0(\la)\}$. We define a bilinear form $\lan\,,\,\ran$ on $W^\la$ by
\[\lan \vf_S,\vf_T \ran \vf_{UV}\equiv \vf_{US}\vf_{TV} \mod \Sc^{\vee \la} \quad \big(S,T\in \CT_0(\la)\big),\]
where $U,V\in \CT_0(\la)$, and the scalar $\lan \vf_S , \vf_T \ran$ does not depend on a choice of $U,V\in \CT_0(\la)$. 
The bilinear form $\lan \,,\,\ran$ is associative, namely we have 
\begin{equation} \label{asso-bi-W}
\lan x\vf ,y \ran =\lan x, y\vf^\ast \ran \quad \text{for }x,y \in W^\la,\,\,\vf\in \Sc. 
\end{equation}
Let $\rad W^\la=\{x \in W^\la\,|\, \lan x,y\ran=0 \text{ for any }y\in W^\la\}$, 
Then $\rad W^\la$ is the $\Sc$-submodule of $W^\la$. 
Put $L^\la=W^\la/\rad W^\la$.
Then it is known by \cite{DJM98} that $L^\la\not=0$ for any $\la\in \vL^+$. 
Assume that $R$ is a field. Then $L^\la$ is an absolutely irreducible module, 
and the set $\{L^\la\,|\,\la \in \vL^+\}$ gives a complete set of non-isomorphic irreducible $\Sc$-modules. 
\para 
We recall some definitions and results in \cite{SW}. 
We fix a positive integer $g\leq r$ and $\Bp=(r_1,\cdots ,r_g) \in \ZZ_{>0}^g$ such that $r_1+\cdots +r_g=r$, 
and set $p_1=0,\,p_i=\sum_{j=1}^{i-1}r_j$ for $i=2,\cdots g$.
 For $\mu=(\mu^{(1)},\cdots ,\mu^{(r)})\in \vL$, 
we define $\alp(\mu)=(n_1,\cdots,n_g)$ and $\ap(\mu)=(a_1,\cdots,a_g)$, 
where $n_k=\sum_{i=1}^{r_k}|\mu^{(p_k+i)}|$ and $a_k=\sum_{i=1}^{k-1}n_i$ for $k=1,\cdots ,g$ with $a_1=0$. 
We define a partial order on $\ZZ_{>0}^g$ by $\Ba=(a_1,\cdots ,a_g) \geq \Bb=(b_1,\cdots ,b_g)$ if $a_i \geq b_i$ for any $i=1,\cdots , g$ 
and we write $\Ba > \Bb$ if $\Ba \geq \Bb$ and $\Ba \not= \Bb$. 
Later we consider the partial order on $\{\ap(\mu)\,|\,\mu \in \vL \}$ by this order.

For $\la \in \vL^+$ and $\mu \in \vL$, we set $\CT_0^\Bp(\la,\mu)=\CT_0(\la,\mu)$ if $\alp(\la)=\alp(\mu)$, and is empty otherwise. 
Moreover we set $\CT_0^\Bp(\la)=\bigcup_{\mu \in \vL}\CT_0^\Bp(\la,\mu)$.
We set 
\begin{align*}
\vS^\Bp=\big(\vL^+\times \{0,1\}\big)\setminus \Big\{(\la,1)&\in \vL^+ \times \{0,1\} \bigm| \CT_0(\la,\mu)=\f \\
	&\text{ for any } \mu\in \vL \text{ such that } \ap(\la)>\ap(\mu) \Big\},
\end{align*}
and define a partial order $\geq$ on $\vS^\Bp$ by $(\la_1,\ve_1)>(\la_2,\ve_2)$ 
	if $\la_1 \trr \la_2$ or if $\la_1=\la_2$ and $\ve_1>\ve_2$. 
For $\e=(\la,\ve)\in \vS^\Bp$, we set 
\begin{align*}
&I(\e)=\begin{cases} \CT_0^\Bp(\la) &\text{if }\ve=0, \\ \\
			\displaystyle \bigcup_{\mu \in \vL \atop \ap(\la)>\ap(\mu)} \CT_0(\la,\mu) &\text{if } \ve=1,
			\end{cases} \\ \\
&J(\e)=\begin{cases} \CT_0^\Bp(\la) &\text{if }\ve =0,\\
					\CT_0(\la)&\text{if }\ve=1, 
			\end{cases}
\end{align*}
\[\ZC^\Bp(\e)=\big\{\vf_{ST}\bigm| (S,T)\in I(\e) \times J(\e) \big\} \qquad \text{for }\e\in \vS^\Bp,\]
and 
\[\ZC^\Bp=\bigcup_{\e\in \vS^\Bp}\ZC^\Bp(\e).\]

Let $\Sp=\Sp(\vL)$ be the $R$-submodule of $\Sc(\vL)$ spanned by $\ZC^\Bp$. 
We also define $(\Sp)^{\vee \e}$ as the $R$-submodule of $\Sp$ spanned by 
\[ \big\{ \vf_{UV} \bigm| (U,V)\in I(\e^\prime) \times J(\e^\prime) 
	\text{ for some } \e^\prime \in \vS^\Bp \text{ such that }\e^\prime > \e \big\}.\]

It is known by \cite[Theorem 2.6]{SW} that $\Sp$ is a standardly based algebra 
with the standard basis $\ZC^\Bp$ in the sence of \cite{DR98}.

By the general theory of standardly based algebra due to \cite{DR98}, we have the following results. 
 For $\e\in \vS^\Bp$, one can consider the standard left $\Sp$-modules $ ^\diamondsuit Z^{\e}$
 with the basis $\big\{\vf^{\e}_T \bigm| T\in I(\e) \big\}$ 
and the standard right $\Sp$-module $Z^{\e}$  with the basis $\big\{\vf_T^{\e} \bigm| T\in J(\e)\big\}$. 
We call them Weyl modules of $\Sp$. 
We  define the bilinear form $\b_{\e}: \,^\diamondsuit Z^{\e} \times Z^{\e} \ra R$ by
\[\b_{\e}\big(\vf_{S}^{\e},\vf_{T}^{\e}\big)\vf_{UV} \equiv \vf_{UT}\vf_{SV} \mod (\Sp)^{\vee \e} 
	\qquad \big(S\in I(\e), T\in J(\e) \big),\]
where $\b_{\e}$ is determined independent of the choice of $U\in I(\e)$ and $V \in J(\e)$. 
The bilinear form $\b_{\e}$ is associative, namely we have 
\begin{equation} \label{asso-bi-Z}
\b_{\e}(\vf x,y)=\b_{\e}(x,y\vf)\quad \text{for }x\in \,^\diamondsuit Z^{\e}, y\in Z^{\e}, \vf \in \Sp.
\end{equation}
Let $\rad Z^{\e}=\big\{ x\in Z^{\e} \bigm| \b_{\e}(y,x)=0 \text{ for any }y\in \,^\diamondsuit Z^{\e} \big\}$. 
Then $\rad Z^{\e}$ is a $\Sp$-submodule of $Z^{\e}$ by associativity of $\b_{\e}$. 
Put $L^{\e}=Z^{\e}/\rad Z^{\e}$. 
Assume that $R$ is a filed. 
Then $L^\e$ is an absolutely irreducible module or zero, and the set 
$\big\{L^{\e} \bigm| \e\in \vS^\Bp \text{ such that }\b_{\e}\not=0 \big\}$ 
is a complete set of non-isomorphic irreducible (right) $\Sp$-modules. 

Later we shall only consider the Weyl modules $Z^\e$ and irreducible modules $L^\e$ of $\Sp$ for $\e$ of the form $(\la,0)$. 
Note that the composition fuctors of $Z^{(\la,0)}$ are isomorphic to $L^{(\mu,0)}$ for some $\mu\in \vL^+$ by \cite[Proposition 3.3 (\roi)]{SW}.
\para
Let $\wh{\Sc}^\Bp$ be the $R$-submodule of $\Sp$ spanned by 
\[\ZC^\Bp \setminus \big\{\vf_{ST}\bigm| S,T\in \CT_0^\Bp(\la) \text{ for some } \la \in \vL^+ \big\}.\]
It is known by \cite{SW} that $\wh{\Sc}^\Bp$ is a two-sided ideal of $\Sp$. 
Thus, we can define the quotient algebra 
\[\oSp=\Sp\big/ \wh{\Sc}^\Bp.\] 
We denote by $\ol{\vf}$  the image of $\vf\in \Sp$ under the natural surjection $\pi: \Sp \ra \oSp$,  
and set 
\[\ol{\ZC}^\Bp =\big\{\ol{\vf}_{ST} \bigm| S,T \in \CT_0^\Bp(\la) \text{ for some } \la \in \vL^+ \big\}.\] 
Then $\ol{\ZC}^\Bp$ is a free $R$-basis of $\oSp$. 
By \cite[Theorem 2.13]{SW}, $\oSp$ turns out to be a cellular algebra with the cellular basis $\ol{\ZC}^\Bp$. 
Hence by the general theory of cellular algebra, the following results hold. 
For $\la \in \vL^+$, we can consider the standard (right) $\oSp$-module $\ol{Z}^\la$ 
with the free $R$-basis $\big\{ \ol{\vf}_T \bigm| T\in \CT_0^\Bp(\la) \big\}$. 
We call it a Weyl module of $\oSp$. 
We define the bilinear form $\lan \,,\, \ran_{\Bp}: \ol{Z}^\la\times \ol{Z}^\la \ra R$ by 
\[\lan \ol{\vf}_{S},\ol{\vf}_T \ran_{\Bp}\, \ol{\vf}_{UV} \equiv \ol{\vf}_{US}\ol{\vf}_{TV} \mod (\oSp)^{\vee \la} 
	\qquad \big(S,T \in \CT_0^\Bp(\la)\big),\] 
where $\lan \,,\,\ran_\Bp $ is determined independent of the choice $U,V \in \CT_0^\Bp(\la)$,  
and $(\oSp)^{\vee \la}$ is the $R$-submodule of $\oSp$ spanned by 
\[\big\{ \ol{\vf}_{ST} \bigm| S,T\in \CT_0^\Bp(\la^\prime) \text{ for some }\la^\prime \in \vL^+ \text{ such that } \la^\prime \trr \la \big\}.\] 
The bilinear form $\lan\,,\,\ran_{\Bp}$ is associative, namely we have 
\begin{equation} \label{asso-bi-barZ}
\lan \ol{x} \,\ol{\vf},\ol{y}\ran_{\Bp} =\lan \ol{x}, \ol{y}\,\ol{\vf}^\ast \ran_{\Bp} \quad \text{for any }\ol{x},\ol{y}\in \ol{Z}^\la, \ol{\vf}\in \oSp.
\end{equation}
Let $\rad \ol{Z}^\la =\big\{ \ol{x} \in \ol{Z}^\la \bigm| \lan \ol{x}, \ol{y} \ran_{\Bp}=0 \text{ for any } \ol{y} \in \ol{Z}^\la \big\}$, 
then $\rad \ol{Z}^\la$ is an $\oSp$-submodule of $\ol{Z}^\la$.  
Put $\ol{L}^\la =\ol{Z}^\la \big/ \rad \ol{Z}^\la$.  
Assume that $R$ is a field. 
Then $\ol{L}^\la$ is an absolutely irreducible module, and the set 
$\big\{\ol{L}^\la \bigm| \la \in \vL^+ \big\}$ is a complete set of non-isomorphic irreducible (right) $\oSp$-modules. 
\para \label{decom-number}
Assuming that $R$ is a field, we set, 
for $\la,\mu \in \vL^+$, 
\begin{align*}
&d_{\la\mu}= \big[W^\la : L^\mu \big] ,\\
&d_{\la\mu}^{(\la,0)}= \big[Z^{(\la,0)} :L^{(\mu,0)}\big] ,\\
&\ol{d}_{\la\mu}= \big[\ol{Z}^\la :\ol{L}^\mu \big], 
\end{align*}
where $\big[W^\la : L^\mu \big] $ is the decomposition number of $L^\mu$ in $W^\la$, and similarly for $\Sp$ and $\oSp$. 
The following theorem was proved in \cite{SW}.
\begin{thm}\cite[Theorem 3.13]{SW}\label{rel-d}
Assume that $R$ is a field. 
For $\la,\mu \in \vL^+$ such that $\alp(\la)=\alp(\mu)$, we have
\[\ol{d}_{\la\mu}=d_{\la\mu}^{(\la,0)}=d_{\la\mu}\]
\end{thm}
\para
For $\mu=(\mu^{(1)},\cdots ,\mu^{(r)}) \in \vL$, we write it in the form $\mu=(\mu^{[1]},\cdots ,\mu^{[g]})$, 
where $\mu^{[i]}=(\mu^{(p_i+1)},\cdots ,\mu^{(p_i +r_{i})})$. 
According to the expression of $\mu$ as above, 
$T=(T^{(1)},\cdots, T^{(r)}) \in \CT_0(\la)$ can be expressed as $T=(T^{[1]},\cdots, T^{[g]})$ with $T^{[i]}=(T^{(p_i +1)},\cdots ,T^{(p_i+r_i)})$. 
By \cite[Lemma 4.3 (\roiii)]{SW}, 
we have a bijection $\CT_0^\Bp(\la,\mu)\simeq \CT_0(\la^{[1]},\mu^{[1]})\times \cdots \times \CT_0(\la^{[g]},\mu^{[g]})$ 
given by the map $T \mapsto (T^{[1]},\cdots, T^{[g]})$. 
Thus we have a bijection $\CT_0^\Bp(\la)\simeq \CT_0(\la^{[1]})\times \cdots \times \CT_0(\la^{[g]})$. 

We write $\Bm=(m_1,\cdots ,m_r)$ in the form $\Bm=(\Bm^{[1]},\cdots,\Bm^{[g]})$, where $\Bm^{[k]}=(m_{p_k+1},\cdots,m_{p_k+r_k})$. 
For each $n_k\in \ZZ_{\geq 0}$, put $\vL_{n_k}=\wt{\CP}_{n_k,r_k}(\Bm^{[k]})$, and $\vL_{n_k}^+=\CP_{n_k,r_k}(\Bm^{[k]})$.
($\vL_{n_k}$ or $\vL_{n_k}^+$ is regarded as the empty set if $n_k=0$.) 
Let $\Sc(\vL_{n_k})$ be the cyclotomic $q$-Schur algebra associated to the Ariki-Koike algebra $\He_{n_k,r_k}$ 
with parameters $q,Q_{p_k+1},\cdots ,Q_{p_k+r_k}$. 
Let $\D_{n,g}$ be the set of $(n_1,\cdots ,n_g)\in \ZZ_{\geq 0}^{g}$ such that $n_1+\cdots +n_g=n$. 
Then we have the following decomposition theorem of $\oSp$ by \cite[Theorem 4.15]{SW}.
\begin{equation} \label{decom-barSp}
\oSp(\vL) \cong \bigoplus_{(n_1,\cdots ,n_g) \in \D_{n,g}} \Sc(\vL_{n_1})\otimes \cdots \otimes \Sc(\vL_{n_g}) \quad \textit{as } R\textit{-algebra},  
\end{equation}
under the isomorphism given by 
\begin{equation} \label{decom-basis}
\ol{\vf}_{ST} \mapsto \vf_{S^{[1]} T^{[1]}} \otimes \cdots \otimes \vf_{S^{[g]} T^{[g]}} \qquad \text{for }S,T \in \CT_0^\Bp(\la).
\end{equation}

Assuming that $R$ is a field, 
for $\la^{[k]} \in \vL_{n_k}$, 
let $W^{\la^{[k]}}$ be the Weyl module of $\Sc(\vL_{n_k})$, and $L^{\la^{[k]}}=W^{\la^{[k]}}/\rad W^{\la^{[k]}}$ be the irreducible module. 
By \cite[Corollary 4.16 ]{SW}, the following properties hold.  
Under the isomorphism in (\ref{decom-barSp}), we have, for $\la,\mu \in \vL^+$, 
\begin{align} 
\label{decom-Weyl}
&\ol{Z}^\la \cong W^{\la^{[1]}}\otimes \cdots \otimes W^{\la^{[g]}},\\
\label{decom-irr}
&\ol{L}^\mu\cong L^{\mu^{[1]}} \otimes \cdots \otimes L^{\mu^{[g]}}, \\
\label{decom-com}
&\big[ \ol{Z}^\la : \ol{L}^\mu \big] =
	\begin{cases} \prod_{k=1}^{g}\big[ W^{\la^{[k]}} : L^{\mu^{[k]}} \big] &\text{if }\alp(\la)=\alp(\mu)\\
		0 &\text{otherwise} \end{cases}.
\end{align}

Under the isomorphism in (\ref{decom-Weyl}), a bilinear form $\lan\,,\,\ran_{\BP}$ on $\ol{Z}^\la$ 
decomposes to a product of bilinear forms on $W^{\la^{[k]}}$ for $k=1,\cdots, g$, namely we have the following lemma. 
\begin{lem} \label{bi-decom}
For $S,T\in \CT_0^\Bp(\la)$, we have
\[\lan \ol{\vf}_S , \ol{\vf}_T \ran_\Bp=\lan \vf_{S^{[1]}}, \vf_{T^{[1]}}\ran \cdots \lan \vf_{S^{[g]}}, \vf_{T^{[g]}}\ran,\]
where $\lan\vf_{S^{[k]}},\vf_{T^{[k]}}\ran $ denotes the bilinear form on $W^{\la^{[k]}}$ for $k=1,\cdots,g$.
\end{lem}
\begin{proof}
Fix $U,V\in \CT_0^\Bp(\la)$. Then by (\ref{decom-basis}) and the definition of the bilinear form on $W^{\la^{[k]}}$, we have 
\begin{align*}
\ol{\vf}_{US}\ol{\vf}_{TV}&=( \vf_{U^{[1]} S^{[1]}}\otimes \cdots \otimes \vf_{U^{[g]} S^{[g]}})
		(\vf_{T^{[1]} V^{[1]}}\otimes \cdots \otimes \vf_{T^{[g]} V^{[g]}})\\
	&=\vf_{U^{[1]} S^{[1]}}\vf_{T^{[1]} V^{[1]}}\otimes \cdots \otimes \vf_{U^{[g]} S^{[g]}}\vf_{T^{[g]} V^{[g]}}\\
	&\equiv \lan\vf_{S^{[1]}},\vf_{T^{^[1]}}\ran \vf_{U^{[1]} V^{[1]}}\otimes \cdots \otimes \lan\vf_{S^{[g]}},\vf_{T^{^[g]}}\ran \vf_{U^{[g]} V^{[g]}}\\
		&\hspace{12em}\mod \Sc(\vL_{n_1})^{\vee\la^{[1]}}\otimes \cdots \otimes \Sc(\vL_{n_k})^{\vee \la^{[g]}} \\
	&= \lan\vf_{S^{[1]}},\vf_{T^{^[1]}}\ran \cdots \lan\vf_{S^{[g]}},\vf_{T^{^[g]}}\ran \vf_{U^{[1]} V^{[1]}}\otimes \cdots \otimes\vf_{U^{[g]} V^{[g]}}\\
	&= \lan\vf_{S^{[1]}},\vf_{T^{^[1]}}\ran \cdots \lan\vf_{S^{[g]}},\vf_{T^{^[g]}}\ran \ol{\vf}_{U V}.
\end{align*}
Since $\Sc(\vL_{n_1})^{\vee \la^{[1]}}\otimes \cdots \otimes \Sc(\vL_{n_k})^{\vee \la^{[g]}}\subset (\oSp)^{\vee \la}$, we see that  
\[\lan \ol{\vf}_S ,\ol{\vf}_T\ran_{\Bp} \ol{\vf}_{UV}\equiv \ol{\vf}_{US}\ol{\vf}_{TV} 
	\equiv  \lan\vf_{S^{[1]}},\vf_{T^{^[1]}}\ran \cdots \lan\vf_{S^{[g]}},\vf_{T^{^[g]}}\ran \ol{\vf}_{U V} \mod (\oSp)^{\vee \la}.\]
The lemma is proved.
\end{proof}
\remark
For the isomorphism in (\ref{decom-Weyl}), we do not need to assume that $R$ is a field. 
But for (\ref{decom-irr}) and (\ref{decom-com}), we need that $R$ is a field.
\begin{thm}\cite[Theorem 4.17]{SW} \label{decom-d}
Assume that $R$ is a field. 
For $\la,\mu \in \vL^+$ such that $\alp(\la)=\alp(\mu)$, we have the following.
\[d_{\la\mu}=\ol{d}_{\la\mu}=\prod_{k=1}^{g}d_{\la^{[k]}\mu^{[k]}},\]
where $d_{\la^{[k]}\mu^{[k]}}=\big[W^{\la^{[k]}} : L^{\mu^{[k]}}\big]$. 
\end{thm}
\section{Decomposition numbers with Jantzen filtration} \label{def-vdecom}

\para 
In the rest of this paper, we assume that $R$ is a discrete valuation ring. Let $\wp$ be a unique maximal ideal of $R$ and $F=R/\wp$ be the residue filed.
Fix $\wh{q},\wh{Q}_1,\cdots ,\wh{Q}_r$ in $R$ and let $q=\wh{q} +\wp\,, Q_1=\wh{Q}_1+\wp\,, \cdots , Q_r=\wh{Q}_r+\wp$ be their canonical images in $F$. 
Moreover let $K$ be the quotient field of $R$. 
Then $(K,R,F)$ is a modular system with parameters. 
Let $\Sc_R=\Sc_R(\Lambda)$ be the cyclotomic $\wh{q}$-Schur algebra over $R$ with parameters $\wh{q}, \wh{Q}_1 , \cdots , \wh{Q}_r$ 
and $\Sc=\Sc(\Lambda)$ be the cyclotomic $q$-Schur algebra over $F$ with parameters $q,Q_1,\cdots ,Q_r$. 
Then $\Sc= \big(\Sc_R+\wp \Sc_R \big)\big/\wp \Sc_R$.

We consider the subalgebra $\Sp_R$ (resp. $\Sp$) of $\Sc_R$ (resp. $\Sc$) and its quotient $\oSp_R$ (resp. $\oSp$) as in the previous section 
with the notation there. Note that the subscript $R$ is used to indicate the objects related to $R$. 

For $\la\in \Lambda^+$, let $W_R^\la$ be the Weyl module of $\Sc_R$. For $i\in \ZZ_{\geq0}$, we set 
\[W_R^\la(i) = \big\{ x\in W_R^\la \bigm| \lan x,y \ran \in \wp^i \text{ for any } y \in W_R^\la \big\} \]
and define
\[W^\la(i) =\big( W_R^\la(i)+ \wp W_R^\la \big) \big/\wp W_R^\la .\]
Then $W^\la = W^\la(0)$ is the Weyl module of $\Sc$, and we have the Jantzen filtration of $W^\la$, 
\[W^\la =W^\la(0)\supset W^\la(1) \supset W^\la(2) \supset \cdots .\]
Similarly, by using the bilinear form $\lan\,,\,\ran_{\Bp}$ on $\ol{Z}_R^\la$,
one can  define the Jantzen filtration of $\oSp$-module $\ol{Z}^\lambda$ 
\[\ol{Z}^\la=\ol{Z}^\la(0)\supset \ol{Z}^\la(1) \supset \ol{Z}^\la(2) \supset \cdots .\]
Moreover for the Weyl module $Z^{(\la ,0)}_R$ of $\Sp_R$, we set 
\[Z_R^{(\la,0)}(i)=\big\{x \in Z_R^{(\la,0)}\bigm| 
	\b_\la (y,x) \in \wp^i \text{ for any }  y \in \,^\diamondsuit Z_R^{(\la,0)} \big\}\]
and  define 
\[Z^{(\la,0)}(i)=\big( Z_R^{(\la,0)}(i) + \wp Z_R^{(\la,0)} \big)\big/ \wp Z_R^{(\la,0)}.\]
Then we have the Jantzen filtration of $Z^{(\la,0)}$
\[Z^{(\la,0)}=Z^{(\la,0)}(0)\supset Z^{(\la,0)}(1)\supset Z^{(\la,0)}(2) \supset \cdots .\]

Since $W^\la $ is  a finite dimentional $F$-vector space, 
one can find a positive integer $k$ such that $W^\la(k^\prime)=W^\la(k) \text{ for any } k^\prime > k$. 
We choose a minimal $k$ in such numbers and set $W^\la (k+1)=0$. Then the Jantzen filtration of $W^\la$ becomes a finite sequence. 
Similarly, Jantzen filtrations of $Z^{(\la,0)}$ and $\ol{Z}^\la$  also become finite sequences. 

We can easily see that 
$W^\la(i)$ \big(resp. $Z^{(\la,0)}(i)$, $\ol{Z}^\la(i)$\big) is a $\Sc$-submodule of $W^\la$ \big(resp. $\Sp$-submodule of $Z^{(\la,0)}$, 
$\oSp$-submodule of $\ol{Z}^\la$\big) 
by associativity of the bilinear form (\ref{asso-bi-W}) \big(resp. (\ref{asso-bi-Z}), (\ref{asso-bi-barZ}) \big).

\para \label{def-vdecome}
Take $\la, \mu \in \vL^+$, and $W^\la=W^\la(0)\supset W^\la(1)\supset \cdots $ be the Jantzen filtration of $W^\la$.
Let $ \big[W^\la(i)/W^\la(i+1):L^\mu \big]$ be the composition multiplicity of $L^\mu$ in $W^\la(i)/W^\la(i+1)$. 
Let $v$ be an indeterminate. We define a polynomial $d_{\la\mu}(v)$ by 
\[d_{\la\mu}(v)=\sum_{i \geq 0} \big[W^\la(i)/W^\la(i+1):L^\mu \big] \cdot v^i .\]
Similarly we define, for $Z^{(\la,0)}$ and $\ol{Z}^\la$,  
\[d_{\la\mu}^{(\la,0)}(v)=\sum_{i \geq 0} \big[Z^{(\la,0)}(i)/Z^{(\la,0)}(i+1) : L^{(\mu,0)}\big] \cdot v^i,\] 
\[\ol{d}_{\la\mu}(v)=\sum_{i \geq 0} \big[ \ol{Z}^\la(i)/\ol{Z}^\la(i+1) : \ol{L}^\mu \big] \cdot v^i .\]
Thus $d_{\la\mu}(v)$, $d_{\la\mu}^{(\la,0)}(v)$ and $\ol{d}_{\la\mu}(v)$ are porynomials whose coefficients are non-negative integers. 
Note that since the Jantzen filtration of $W^\la $, etc. are finite sequences, these summations are finite sums.
We call $d_{\la \mu}(v)$ (resp. $d_{\la\mu}^{(\la,0)}(v)$, $\ol{d}_{\la\mu}(v)$) 
	decomposition number with Jantzen filtration of $\Sc$ (resp. $\Sp$, $\oSp$). 
We also call them $v$-decomposition numbers as they coincide at $v=1$ with  decomposition numbers given in \ref{decom-number}.

We have the following relation between $d_{\la\mu}^{(\la,0)}(v)$ and $\ol{d}_{\la\mu}(v)$.

\begin{prop} \label{vd-bar-p}
For $\la,\mu \in \vL$, we have 
\begin{enumerate}
\item \label{q1} If $\alp(\la)\not= \alp(\mu)$, then $\ol{d}_{\la\mu}(v)=d_{\la\mu}^{(\la,0)}(v)=0$.
\item \label{q2} $\big[ \ol{Z}^\la(i)/\ol{Z}^\la(i+1) : \ol{L}^\mu \big] =\big[ Z^{(\la,0)}(i)/ Z^{(\la,0)}(i+1):L^{(\mu,0)}\big] 
	\quad \textit{for any }i \geq 0$. \\
Hence we have $\ol{d}_{\la\mu}(v)=d_{\la\mu}^{(\la,0)}(v)$.
\end{enumerate}
\end{prop}

\begin{proof}
(\ref{q1}) is  clear since $d_{\la\mu}^{(\la,0)}=\ol{d}_{\la\mu}=0$ by \cite[Poposition 3.3]{SW}.

Recall that $\ol{Z}^\la \cong Z^{(\la,0)}$ and $\ol{L}^\mu \cong L^{(\mu,0)}$ as $\Sp$-modules by \cite[Lemma 3.2]{SW}. 
By definition, we have $\b_\la\big(\vf_S^{(\la,0)},\vf_T^{(\la,0)}\big)=\lan \ol{\vf}_T , \ol{\vf}_S \ran_{\Bp}$ 
	for any $S,T \in \CT_0^\Bp(\la)$. 
Then under the isomorphism $\ol{Z}^\la \cong Z^{(\la,0)}$, 
the Jantzen filtration of $\ol{Z}^\la$ coincides with that of $Z^{(\la,0)}$. So (\ref{q2}) is proved.
\end{proof}

\para
Next, we consider the relation between $d_{\la\mu}^{(\la,0)}(v)$ and $d_{\la\mu}(v)$. In order to see this we prepare two lemmas.
Recall that there exists an injective $\Sp$-homomorphism $f_\la :Z^{(\la,0)} \hra W^\la$ such that 
$f_\la(\vf_{T}^{(\la,0)})=\vf_T$ for $T\in \CT_0^\Bp(\la)$ 
by \cite[Lemma 3.5]{SW}
and that $Z^{(\la,0)}\otimes_{\Sp} \Sc \cong W^\la$ as $\Sc$-module by \cite[Proposition 3.6]{SW}. 
Let $\io_i: Z^{(\la,0)}(i)\hra Z^{(\la,0)}$ be an inclusion map. 
Then $(\io_i \otimes \id_\Sc) \big(Z^{(\la,0)}(i)\otimes_{\Sp}\Sc\big)$ is the $\Sc$-submodule of $Z^{(\la,0)}\otimes_{\Sp} \Sc$. 
Similar results hold also for $R$.
We have the following.

\begin{lem} \label{Sp-include}
Let $\la\in \vL^+ $. For any $i \geq 0$, we have 
\[f_\la^{-1}\big(W^\la(i)\big) =Z^{(\la,0)}(i)\]
\end{lem}
\begin{proof}
By definition, we see that $\b \big(\vf_T^{(\la,0)},\vf_S^{(\la,0)}\big)=\lan \vf_{S},\vf_T \ran$ for any $S,T \in \CT_0^\Bp(\la)$, 
and that $\lan \vf_S,\vf_T \ran =0 $ if $S\in \CT_0^\Bp(\la),\,\,T \in \CT_0(\la) \setminus \CT_0^\Bp(\la)$. 
Then for $x \in Z_R^{(\la,0)}$, we have
\begin{align*}
x \in Z_R^{(\la,0)}(i) &\LRa \b_\la \big(\vf_T^{(\la,0)},x \big) \in \wp^i \quad \text{for any }T\in \CT_0^\Bp(\la)\\
	&\LRa \lan f_\la(x) , \vf_T \ran \in \wp^i \quad \text{for any } T \in \CT_0(\la)\\
	&\LRa f_\la(x) \in W_R^\la(i)
\end{align*}

By taking the quotient, we obtain the lemma.  
\end{proof}

\begin{lem} \label{S-include}
Let $\la\in \vL^+$. For any $i\geq 0$, we have 
\[(\io_i \otimes \id_\Sc)\big(Z^{(\la,0)}(i)\otimes_{\Sp}\Sc\big)\subset W^\la(i)\]
under the isomorphism $Z^{(\la,0)}\otimes_{\Sp}\Sc\cong W^\la$.
\end{lem}
\begin{proof}
Recall that any element of $Z_R^{(\la,0)}$ can be written in the form $\vf_{T^\la}^{(\la,0)}\cdot \p$ with $\p\in \Sp_R$.
Moreover it follows from \cite[Proposition 3.6]{SW} that, 
under the isomorphism $g_\la:Z_R^{(\la,0)}\otimes_{\Sp}\Sc_R \xrightarrow{\sim}W_R^\la$, 
we have $g_\la\big(\vf_{T^\la}^{(\la,0)}\cdot \p \otimes \vf\big)=\vf_{T^\la}\cdot \p\vf$ for $\p\in \Sp_R,\,\vf\in \Sc_R$. 
This is true also for $Z^{(\la,0)}$ and $W^\la$. Thus in order to show the lemma, 
it is enough to prove the following.

\begin{flushleft}
(\thethm.1) Suppose that $\vf_{T^\la}^{(\la,0)}\p \in Z^{(\la,0)}(i)$ for $\p \in \Sp$. 
Then we have $\vf_{T^\la}\p\vf \in W^\la(i)$\\ \hspace{3em} for any $\vf\in \Sc$.
\end{flushleft}
Now take $ \vf_{T^\la}^{(\la,0)}\p \in Z_R^{(\la,0)}$.
If $\vf_{T^\la}^{(\la,0)}\cdot \p \in Z_R^{(\la,0)}(i)$, then $ \b_\la\big(x\,, \,\vf_{T^\la}^{(\la,0)}\p\big)\in \wp^i$ 
for any $ x \in \,^\diamondsuit Z_R^{(\la,0)}$. 
This implies that $\lan \vf_{T^\la}\p\,,\,y\ran \in \wp^i$ for any $ y \in W_R^\la$ by a similar argument as the proof of Lemma \ref{Sp-include}.

Since $\lan \vf_{T^\la}\p\vf \,,\, y\ran =\lan \vf_{T^\la}\p \,,\, y\vf^\ast\ran $ for any $y\in W_R^\la$ and any $\vf \in \Sc_R$, we see that 
$\vf_{T^\la}^{(\la,0)}\p\in Z_R^{(\la,0)}(i) $ implies that $\vf_{T^\la}\p\vf \in W_R^\la(i)$. 
By taking the quotient, we obtain (\thethm.1). Thus the lemma is proved.
\end{proof}

These two lemmas imply the following proposition about the relation between $d_{\la\mu}^{(\la,0)}(v)$ and $d_{\la\mu}(v)$.

\begin{prop} \label{vd-sp-s}
Let $\la, \mu \in \vL^+ $ be such that $\alp(\la)=\alp(\mu)$. Then for any $i\geq 0$, we have 
\[\big[ Z^{(\la,0)}(i)\big/ Z^{(\la,0)}(i+1) : L^{(\mu,0)}\big]=\big[ W^\la(i)\big/ W^\la(i+1) : L^\mu \big].\]
Hence we have $d_{\la\mu}^{(\la,0)}(v)=d_{\la\mu}(v)$ if $\alp(\la)=\alp(\mu)$.
\end{prop}
\begin{proof}
Fix $\la, \mu \in \vL^+$ such that $\alp(\la)=\alp(\mu)$, and an integer $i\geq 0$.
Thanks to Lemma \ref{Sp-include}, we have the following result by similar arguments as in the proof of \cite[Proposition 3.12]{SW}.
\begin{equation} \label{geq}
\big[Z^{(\la,0)}(i)\big/ Z^{(\la,0)}(i+1):L^{(\mu,0)}\big] \geq [W^\la(i)\big/ W^\la(i+1):L^\mu \big].
\end{equation}

Conversly, thanks to Lemma \ref{S-include}, we have the following result by similar arguments as in the proof of \cite[Proposition 3.11]{SW}. 
\begin{equation} \label{leq}
\big[Z^{(\la,0)}(i):L^{(\mu,0)}\big] \leq \big[W^\la(i):L^\mu\big].
\end{equation}

We remark that this does not imply dilectry 
\[\big[Z^{(\la,0)}(i)\big/Z^{(\la,0)}(i+1):L^{(\mu,0)}\big] \leq \big[W^\la(i)\big/W^\la(i+1):L^\mu\big] \] 
since we can not see whether $(\io_{i+1}\otimes \id_\Sc)\big(Z^{(\la,0)}(i+1)\otimes_{\Sp}\Sc\big)=W^\la(i+1)$ or not. 
Instead, we argue as follows. Let 
\[W^\la=W^\la(0)\supset W^\la(1)\supset \cdots \supset W^\la(k)\supsetneqq W^\la(k+1)=0\]
\[Z^{(\la,0)}=Z^{(\la,0)}(0)\supset Z^{(\la,0)}(1) \supset \cdots \supset Z^{(\la,0)}(l) \supsetneqq Z^{(\la,0)}(l+1)=0\]
be the Jantzen filtrations of $W^\la$ and $Z^{(\la,0)}$ respectively. 
Then we have 
\[(\io_{k+1}\otimes \id_\Sc )\big(Z^{(\la,0)}(k+1)\otimes_{\Sp}\Sc\big) \subset W^\la(k+1)=0\]
 by Lemma \ref{S-include}. 
This implies that $Z^{(\la,0)}(k+1)=0$ 
since $(\io\otimes \id_\Sc)(M\otimes_\Sp \Sc)\not=0$ for any non-zero submodule $M$ of $Z^{(\la,0)}$ and the inclusion map 
$\iota: M\hookrightarrow Z^{(\la,0)}$ by \cite[Lemma 3.8 (\roii)]{SW}.
So we have $l\leqq k$. 

Now, if $L^\mu$ is a composition factor of $W^\la(i)\big/W^\la(i+1)$, then we have 
\[1\leq \big[W^\la(i)\big/ W^\la(i+1) : L^\mu \big] \leq \big[ Z^{(\la,0)}(i)\big/Z^{(\la,0)}(i+1) : L^{(\mu,0)}\big]\]
by (\ref{geq}). Hence, we have $Z^{(\la,0)}(i)\not=0$. 
This implies that $i\leqq l$ if $L^\mu $ is a composition factor of $W^\la(i)$.
In particular, $\big[W^\la(l+1) : L^\mu \big]=0$. Thus we have 
\[ \big[Z^{(\la,0)}(l)\big/ Z^{(\la,0)}(l+1) : L^{(\mu,0)}\big] = \big[ Z^{(\la,0)}(l):L^{(\mu,0)}\big] \]
\[ \big[W^\la(l)\big/ W^\la(l+1) : L^\mu \big]=\big[W^\la(l):L^\mu\big].\]
Combining these equalities with (\ref{geq}) and (\ref{leq}), we have 
\[\big[ Z^{(\la,0)}(l)\big/ Z^{(\la,0)}(l+1) : L^{(\mu,0)}\big] =\big[ W^\la(l)\big/ W^\la(l+1) : L^\mu \big],\]
and so 
\begin{equation} \label{i=l}
\big[ Z^{(\la,0)}(l) : L^{(\mu,0)}\big] =\big[ W^\la(l) : L^\mu \big].
\end{equation}

Next we consider the case where $i=l-1$. Note that
\[ \big[Z^{(\la,0)}(l-1)\big/ Z^{(\la,0)}(l) : L^{(\mu,0)}\big] 
	=\big[Z^{(\la,0)}(l-1) : L^{(\mu,0)}\big] -\big[Z^{(\la,0)}(l) : L^{(\mu,0)}\big] ,\]
\[\big[W^\la(l-1)\big/ W^\la(l) : L^\mu \big]=\big[W^\la(l-1) : L^\mu \big]-\big[W^\la(l) : L^\mu \big]. \]
Combined with (\ref{geq}), (\ref{leq}) and (\ref{i=l}), we have
\[\big[ Z^{(\la,0)}(l-1)\big/ Z^{(\la,0)}(l) : L^{(\mu,0)}\big] =\big[ W^\la(l-1)\big/ W^\la(l) : L^\mu \big]\]
and so 
$\big[ Z^{(\la,0)}(l-1) : L^{(\mu,0)}\big] =\big[ W^\la(l-1) : L^\mu \big]$. 
Therefore by backward induction on $l$, we obtain the proposition.
\end{proof}

Combining Proposition \ref{vd-bar-p} and Proposition \ref{vd-sp-s}, we have the following theorem.

\begin{thm} \label{th-vdecom}
For any $\la,\mu \in \vL^+$ such that $\alp(\la)=\alp(\mu)$, we have
\[\ol{d}_{\la\mu}(v)=d_{\la\mu}^{(\la,0)}(v)=d_{\la\mu}(v).\]
\end{thm}

If we specialize $v=1$, the theorem reduceds to Theorem \ref{rel-d}.
\para  \label{basis-W(i)}
For later use, we shall consider the basis of $W_R^\la(i)$, following the computation in the proof of \cite[Lemma 5.30]{M-book}. 
Let $\pi$ be the generator of $\wp$, namely $\wp=(\pi)$, and $\nu_\wp$ be the valuation map on $R$. 
Let $G^\la=\big( \lan \vf_S, \vf_T \ran\big)_{S,T\in \CT_0(\la)}$ be the Gram matrix of $W_R^\la$. 
Since $R$ is a PID, there exist $P,Q \in \GL_N(R)$ (where $N=|\CT_0(\la)|$) such that 
$P G^\la Q=\diag (d_{S_1},d_{S_2},\cdots, d_{S_N})$, where $d_{S_k}\in R$ and $\{S_1,\cdots ,S_N\}=\CT_0(\la)$.
Let $ P=\big(p_{ST}\big)_{S,T\in \CT_0(\la)},\,\, Q=\big(q_{ST}\big)_{S,T\in \CT_0(\la)}$ 
and we define, for $S,T\in \CT_0(\la)$,
\[f_S=\sum_{S^\prime\in \CT_0(\la)}p_{SS^\prime}\vf_{S^\prime}, \qquad 
	g_T=\sum_{T^\prime\in \CT_0(\la)}q_{T^\prime T}\vf_{T^\prime}.\]
Since both $P$ and $Q$ are regular matrices, $\{f_S\,|\,S\in \CT_0(\la) \}$ and $\{g_T\,|\, T\in \CT_0(\la)\}$ 
are basis of $W_R^\la$ respectively. 
Moreover we have $\diag(d_{S_1},\cdots ,d_{S_N})=P G^\la Q =\big(\lan f_S , g_T \ran \big)_{S,T\in \CT_0(\la)}$ 
by definition. Thus we have
\begin{equation} \label{orthogonal}
\lan f_S, g_T \ran =\d_{ST}d_S \quad \big(S,T\in \CT_0(\la) \big)
\end{equation}
where $\d_{ST}=1$ if $S=T$ and $\d_{ST}=0$ otherwise. 
For $x=\sum_{S\in \CT_0(\la)} r_S f_S  \in W_R^\la \quad(r_S\in R)$, we have 
\begin{align*}
x\in W_R^\la(i) &\LRa \lan x, g_T \ran \in \wp^i \qquad \text{for any } T\in \CT_0(\la) \\
	&\LRa r_Td_T \in \wp^i  \qquad \text{for any } T\in \CT_0(\la) \quad (\text{by }(\ref{orthogonal})) \\
	&\LRa \nu_\wp(r_Td_T)=\nu_\wp(r_T)+\nu_\wp(d_T) \geq i \qquad \text{for any }T\in \CT_0(\la).
\end{align*}
It follows from this that $W_R^\la(i)$ is a free $R$-module with basis
\begin{equation} \label{W(i)-basis}
\big\{f_T \bigm| T\in \CT_0(\la),\,\nu_\wp(d_T)\geq i \big\}\cup 
	\big\{\pi^{i-\nu_\wp(d_T)}f_T \bigm| T\in \CT_0(\la),\,\nu_\wp(d_T)<i \big\}.
\end{equation}

\para
We consider the Jantzen filtration of $W^{\la{[k]}}$ ($1\leq k \leq g$) as in the case of $W^\la$ 
and use the notation similar to the case of $W^\la$. 
Since we see that $W_R^{\la^{[k]}}(i_k)$ \,($i_k \geq 0$) is a free $R$-module (see \ref{basis-W(i)}.), 
$W_R^{\la^{[1]}}(i_1)\otimes \cdots \otimes W_R^{\la^{[g]}}(i_g)$ \,($(i_1,\cdots ,i_g) \in Z_{\geq0}^g$) 
becomes the submodule of $W_R^{\la^{[1]}}\otimes \cdots \otimes W_R^{\la^{[g]}}$.

For $1\leq k \leq g$, 
let $\big\{f_{S^{[k]}} \,|\, S^{[k]}\in \CT_0(\la^{[k]}) \big\}$ and $\big\{g_{T^{[k]}}\,|\, T^{[k]}\in \CT_0(\la^{[k]}) \big\}$ 
be the bases of $W_R^{\la^{[k]}}$ as of $W_R^\la$ in \ref{basis-W(i)}. 
For $S,T\in \CT_0^\Bp(\la)$, we define $\ol{f}_S:=f_{S^{[1]}}\otimes \cdots \otimes f_{S^{[g]}}$ and 
$\ol{g}_T:=g_{T^{[1]}}\otimes \cdots \otimes g_{T^{[g]}}$. 
Then $\big\{ \ol{f}_S \bigm| S\in \CT_0^\Bp(\la) \big\}$ and $\big\{ \ol{g}_T \bigm| T\in \CT_0^\Bp(\la) \big\}$ 
turn out to be the bases of $\ol{Z}_R^\la$. 
By Lemma \ref{bi-decom} and (\ref{orthogonal}), we have 
\begin{equation}
\lan \ol{f}_S, \ol{g}_T \ran_{\Bp}=\d_{ST}d_{T^{[1]}}\cdots d_{T^{[g]}} \qquad \text{for } S,T\in \CT_0^\Bp(\la) .
\end{equation}
We set $d_T=d_{T^{[1]}}\cdots d_{T^{[g]}}$. Then we have the following result by a similar argument as in \ref{basis-W(i)}. 
$\ol{Z}_R^\la(i)$ is a free $R$-module with basis 
\begin{equation}
\big\{ \ol{f}_T \bigm| T\in \CT_0^\Bp(\la),\,\nu_\wp(d_T)\geq i \big\} \cup 
	\big\{ \pi^{i-\nu_\wp(d_T)} \ol{f}_T \bigm| T\in \CT_0^\Bp(\la),\,\nu_\wp(d_T) <i \big\}.
\end{equation}  
Recall that $\D_{i,g}$ is the set of $(i_1,\cdots ,i_g)\in \ZZ_{\geq 0}^g$ such that $i_1+\cdots +i_g=i$. 
Then we have the following proposition.

\begin{prop} \label{i-decom-prop}
Let $\la\in \vL^+$ and $i \geq 0$. Under the isomorphism $\ol{Z}^\la_R\cong W_R^{\la^{[1]}}\otimes \cdots \otimes W_R^{\la^{[g]}}$, 
we have 
\begin{equation} \label{i-decom}
 \ol{Z}_R^\la(i) = \sum_{(i_1,\cdots ,i_g) \in \D_{i,g}} W_R^{\la^{[1]}}(i_1)\otimes \cdots \otimes W_R^{\la^{[g]}}(i_g)
\end{equation}
\end{prop}

\begin{proof}
First we show that the right hand side is contained in the left hand side. 
Take $x=x^{[1]}\otimes \cdots \otimes x^{[g]} \in W_R^{\la^{[1]}}(i_1) \otimes \cdots \otimes W_R^{\la^{[g]}}(i_g)$ such that $i_1+\cdots +i_g=i$. 
By Lemma \ref{bi-decom}, we have \vspace{-1em}
\begin{align*}
\lan x , \ol{\vf}_T \ran_{\Bp}&=\lan x^{[1]}, \vf_{T^{[1]}}\ran \cdots \lan x^{[g]}, \vf_{T^{[g]}} \ran \\
& \in \wp^{i_1}\cdots \wp^{i_g}=\wp^i \qquad \text{for any }T\in \CT_0^\Bp(\la).
\end{align*}
Thus we have $x \in \ol{Z}_R^\la(i)$. 

Then in order to show the equality, we have only to show that the basis element of 
$\ol{Z}_R^{\la}(i)$ is contained in the right hand side of (\ref{i-decom}).
First, we consider  $\ol{f}_T$ such that $\nu_\wp(d_T)\geq i$.
Since $\nu_\wp(d_T)=\nu_\wp(d_{T^{[1]}})+\cdots +\nu_\wp(d_{T^{[g]}})$,  
one can find $(i_1,\cdots ,i_g)\in \ZZ_{\geq 0}^g $ such that 
	$i_1+ \cdots +i_g=i$ and that $\nu_{\wp}(d_{T^{[k]}})\geq i_k $ for $k=1,\cdots, g$. 
Then 
\[\ol{f}_T=f_{T^{[1]}}\otimes \cdots \otimes f_{T^{[g]}} \in W_R^{\la^{[1]}}(i_1)\otimes \cdots \otimes W_R^{\la^{[g]}}(i_g), \] 
 and so $\ol{f}_T$ is contained in the right hand side of (\ref{i-decom}). 

Next we consider $\ol{f}_T$ such that $\nu_\wp(d_T) <i$. 
Then one can find $(i_1,\cdots ,i_g)$ such that 
	$i_1+\cdots +i_g=i$ and that $\nu_{\wp}(d_{T^{[k]}})\leq i_k$ for $k=1,\cdots ,g$. 
Therefore $\pi^{i-\nu_\wp(d_T)}\ol{f}_T= \big(\pi^{i_1-\nu_\wp(d_{T^{[1]}})}f_{T^{[1]}}\big) \otimes \cdots \otimes 
	\big(\pi^{i_g-\nu_\wp(d_{T^{[g]}})}f_{T^{[g]}}\big) $ is also contained in the right hand side of (\ref{i-decom}). 
The proposition is proved.
\end{proof}
We have the following corollary. 

\begin{cor} \label{cor-i-decom}
For $\la\in \vL^+$ and $i \geq 0$, under the isomorphism $\ol{Z}^\la\cong W^{\la^{[1]}}\otimes \cdots \otimes W^{\la^{[g]}}$, 
we have 
\[ \ol{Z}^\la(i) = \sum_{(i_1,\cdots ,i_g) \in \D_{i,g}} W^{\la^{[1]}}(i_1)\otimes \cdots \otimes W^{\la^{[g]}}(i_g) \]
\end{cor}
\begin{proof}
By definition, we have 
\[\ol{Z}^\la(i)=\big( \ol{Z}_R^\la(i) + \wp \ol{Z}_R^\la \big)\big/ \wp \ol{Z}_R^\la 
	\,\cong\, \ol{Z}_R^\la (i) \big/ \big(\ol{Z}_R^\la(i)\cap \wp \ol{Z}_R^\la\big)\]
and 
\begin{align*}
&W^{\la^{[1]}}(i_1)\otimes \cdots \otimes W^{\la^{[g]}}(i_g) \\
	& \hspace{3em}=\big(W^{\la^{[1]}}_R(i_1)+ \wp W^{\la^{[1]}}_R \big)\big/ \wp W^{\la^{[1]}}_R \otimes \cdots 
		\otimes \big( W^{\la^{[g]}}_R(i_g)+ \wp W^{\la^{[g]}}_R\big) \big/ \wp W^{\la^{[g]}}_R \\
	& \hspace{3em} \cong  W^{\la^{[1]}}_R(i_1)\big/ \big(W^{\la^{[1]}}_R(i_1) \cap \wp W^{\la^{[1]}}_R\big) 
		\otimes \cdots \otimes W^{\la^{[g]}}_R(i_g) \big/ \big( W^{\la^{[g]}}_R(i_1)\cap \wp W^{\la^{[g]}}\big).
\end{align*}
By Proposition \ref{i-decom-prop}, we have a surjective map
\[\F: \ol{Z}_R^\la(i) \ra \sum_{(i_1,\cdots ,i_g) \in \D_{i,g}} W^{\la^{[1]}}(i_1)\otimes \cdots \otimes W^{\la^{[g]}}(i_g). \]
We claim that $\Ker \F =\ol{Z}_R^\la(i)\cap \wp \ol{Z}^\la$. Then the claim implies the corollary. So we shall show the claim. 
By definition, it is clear that $\Ker \F$ is contained in $\ol{Z}^\la_R(i)\cap \wp \ol{Z}^\la$. 
Take $x=\sum_{T\in \CT_0^\Bp(\la)} r_T \ol{f}_T \in \ol{Z}^\la_R(i)\cap \wp \ol{Z}^\la$. Then $r_T \in \wp$ for any $T\in \CT_0^\Bp(\la)$. 
If $\nu_\wp(d_T)\geq i$, then $\ol{f}_T=f_{T^{[1]}}\otimes \cdots \otimes f_{T^{[g]}}$ is contained in 
$W_R^{\la^{[1]}}(i_1)\otimes \cdots \otimes W_R^{\la^{[g]}}(i_g)$ for some $(i_1,\cdots ,i_g) \in \D_{i,g}$ 
by the proof of Proposition \ref{i-decom-prop}. So we have $r_T \ol{f}_T \in \Ker \F$ for $T\in \CT_o^\Bp(\la)$ such that $\nu_\wp(d_T)\geq i$. 
If $\nu_\wp(d_T)<i$, then $r_T\ol{f}_T=r^\prime_T \pi^{i-\nu_\wp(d_T)}\ol{f}_T$ for some $r_T^\prime \in R$ since $x\in \ol{Z}_R^\la(i)$. 
By the proof of Proposition \ref{i-decom-prop}, for some $(i_1,\cdots ,i_g)\in \D_{i,g}$, 
$\pi^{i-\nu_\wp(d_T)}\ol{f}_T=\big(\pi^{i_1-\nu_\wp(d_{T^{[1]}})}f_{T^{[1]}}\big)\otimes \cdots \otimes 
	\big(\pi^{i_g-\nu_\wp(d_{T^{[g]}})}f_{T^{[g]}}\big)$ is contained in $W_R^{\la^{[1]}}(i_1)\otimes \cdots \otimes W_R^{\la^{[g]}}(i_g)$. 
Moreover one can find at least one $k$ such that $\nu_\wp(d_{T^{[k]}})<i_k$. Then the image of $\pi^{i_k-\nu_\wp(d_{T^{[k]}})}f_{T^{[k]}}$ 
in $W_R^{\la^{[k]}}(i_k)\big/(W_R^{\la^{[k]}}(i_k)\cap \wp W^{\la^{[k]}})$ is  zero. Hence for $T\in \CT_0^\Bp$ such that $\nu_\wp(d_T)<i$, 
$r_T \ol{f}_T$ is also contained in $\Ker \F$. Now the claim is proved, and the corollary follows.
\end{proof}

By using the corollary, we show the following lemma. 
\begin{lem} \label{i-decom lem}
Let $\la, \mu \in \vL^+$ be such that $\alp(\la)=\alp(\mu)$. For any $i\geq 0$, we have
\[\big[\ol{Z}^\la(i)\big/ \ol{Z}^\la(i+1) : \ol{L}^\mu \big] 
	=\sum_{(i_1,\cdots ,i_g) \in \D_{i,g}} 
		\prod_{k=1}^{g} \big[W^{\la^{[k]}}(i_k)\big/ W^{\la^{[k]}}(i_{k+1}) : L^{\mu^{[k]}} \big]\]
\end{lem}
\begin{proof}
By Corollary \ref{cor-i-decom}, we have

{\footnotesize
\begin{align}&\ol{Z}^\la(i)\,\big/\,\ol{Z}^\la(i+1) 
 = \Big(\sum_{(i_1,\cdots ,i_g) \in \D_{i,g}}W^{\la^{[1]}}(i_1)\otimes \cdots \otimes W^{\la^{[g]}}(i_g)\Big)\Big/ \ol{Z}^\la(i+1) \label{qwq} \\
 & \hspace{-1em} 
=\sum_{(i_1,\cdots ,i_g) \in \D_{i,g}} \big( W^{\la^{[1]}}(i_1)\otimes \cdots \otimes W^{\la^{[g]}}(i_g)\big) \Big/
	\Big(\ol{Z}^\la (i+1) \cap \big( W^{\la^{[1]}}(i_1)\otimes \cdots \otimes W^{\la^{[g]}}(i_g)\big) \Big). \notag
\end{align}}
If $(i_1,\cdots , i_g)\not= (j_1,\cdots , j_g)$ such that $i_1+\cdots + i_g= j_1+\cdots + j_g=i$, then 
\[\big( W^{\la^{[1]}}(i_1)\otimes \cdots \otimes W^{\la^{[g]}}(i_g)\big) 
	\cap \big( W^{\la^{[1]}}(j_1)\otimes \cdots \otimes W^{\la^{[g]}}(j_g)\big)  \subset W^{\la^{[1]}}(k_1)\otimes \cdots \otimes W^{\la^{[g]}}(k_g),\]
where $k_l=\max\{i_l, j_l\}$. Since $(i_1,\cdots , i_g)\not= (j_1,\cdots , j_g)$, $k_1+\cdots +k_g\geq i+1$. 
Hence we have 
\[\big( W^{\la^{[1]}}(i_1)\otimes \cdots \otimes W^{\la^{[g]}}(i_g)\big) 
	\cap \big( W^{\la^{[1]}}(j_1)\otimes \cdots \otimes W^{\la^{[g]}}(j_g)\big)  \subset \ol{Z}^\la(i+1).\]
It follows from this, we see that the sum in (\ref{qwq}) is a direct sum.

For $(i_1,\cdots ,i_g) \in \D_{i,g}$, we consider a surjective $\oSp$-homomorphism
{\small \[\Psi :W^{\la^{[1]}}(i_1)\otimes \cdots \otimes W^{\la^{[g]}}(i_g) \ra
	W^{\la^{[1]}}(i_1)\big/ W^{\la^{[1]}}(i_1+1) \otimes \cdots \otimes W^{\la^{[g]}}(i_g)\big/ W^{\la^{[g]}}(i_g+1) \] }
Then we have $\Ker \Psi =\ol{Z}^\la(i+1)\cap \big(W^{\la^{[1]}}(i_1)\otimes \cdots \otimes W^{\la^{[g]}}(i_g) \big)$ 
under the setting in Corollary \ref{cor-i-decom}. 
By noting that (\ref{qwq}) is a direct sum, we have 
\[\ol{Z}^\la(i) \big/ \ol{Z}^\la(i+1) \cong \bigoplus_{(i_1,\cdots ,i_g) \in \D_{i,g}}
	\Big( W^{\la^{[1]}}(i_1)\big/ W^{\la^{[1]}}(i_1+1) \otimes \cdots \otimes W^{\la^{[g]}}(i_g)\big/ W^{\la^{[g]}}(i_g+1)\Big)\]
Since $\ol{L}^\mu \cong L^{\mu^{[1]}}\otimes \cdots \otimes L^{\mu^{[g]}}$, we have 
{\footnotesize
\begin{align*}
\Big[ \ol{Z}^\la (i) \big/ \ol{Z}^\la(i+1) : \ol{L}^\mu \Big] 
&=\sum_{(i_1,\cdots ,i_g) \in \D_{i,g}} 
	\Big[ W^{\la^{[1]}}(i_1)\big/ W^{\la^{[1]}}(i_1+1) \otimes \cdots \otimes W^{\la^{[g]}}(i_g)\big/ W^{\la^{[g]}}(i_g+1) : \ol{L}^\mu \Big]	\\
&= \sum_{(i_1,\cdots ,i_g) \in \D_{i,g}} \prod_{k=1}^{g}
	\Big[W^{\la^{[k]}}(i_k)\big/ W^{\la^{[k]}}(i_k+1) : L^{\mu^{[k]}} \Big]
\end{align*} }
The lemma is proved
\end{proof}

We define $v$-decomposition numbers of $\Sc(\vL_{n_k})$ for $k=1,\cdots , g$ by
\[d_{\la^{[k]} \mu^{[k]}}(v):= \sum_{i_k \geq 0} \Big[ W^{\la^{[k]}}(i_k)\big/ W^{\la^{[k]}}(i_k+1) : L^{\mu^{[k]}}\Big] \cdot v^i \]
as in the case of $\Sc(\vL)$. Then we have the following theorem. 

\begin{thm} \label{th-decom-bar-v}
For $\la,\mu \in \vL^+$ such that $\alp(\la)=\alp(\mu)$, we have 
\[d_{\la\mu}(v)=\ol{d}_{\la\mu}(v) =\prod_{k=1}^{g} d_{\la^{[k]} \mu^{[k]}}(v).\]
\end{thm}
\begin{proof}
The first equality follows from Theorem \ref{th-vdecom}. So we prove the second equality. 
By Lemma \ref{i-decom lem}, we have 
\begin{align*}
\ol{d}_{\la\mu}(v)&= \sum_{i\geq 0} \big[ \ol{Z}^\la(i) \big/ \ol{Z}^\la(i+1) : \ol{L}^\mu \big] \cdot v^i\\
&= \sum_{i\geq 0} \Big( \sum_{(i_1,\cdots ,i_g) \in \D_{i,g}} \prod_{k=1}^{g}
	\big[W^{\la^{[k]}}(i_k)\big/ W^{\la^{[k]}}(i_k+1) : L^{\mu^{[k]}} \big] \Big) \cdot v^i \\
&=\sum_{i\geq 0} \sum_{(i_1,\cdots ,i_g) \in \D_{i,g}} \Big(\prod_{k=1}^{g}
	\big[W^{\la^{[k]}}(i_k)\big/ W^{\la^{[k]}}(i_k+1) : L^{\mu^{[k]}} \big] \cdot v^{i_k} \Big)\\
&=\sum_{(i_1,\cdots ,i_g)\in \ZZ_{\geq 0}^g} \Big(\prod_{k=1}^{g}
	\big[W^{\la^{[k]}}(i_k)\big/ W^{\la^{[k]}}(i_k+1) : L^{\mu^{[k]}} \big] \cdot v^{i_k} \Big)\\
&= \prod_{k=1}^{g} \sum_{i_k\geq 0}\big[W^{\la^{[k]}}(i_k)\big/ W^{\la^{[k]}}(i_k+1) : L^{\mu^{[k]}} \big] \cdot v^{i_k} \\
&=\prod_{k=1}^{g}d_{\la^{[k]}\mu^{[k]}}(v).
\end{align*}
This proves the theorem.
\end{proof}
\section{$v$-Decomposition numbers for Ariki-Koike algebras}
We keep the notation in the previous section. 
We consider the $v$-decomposition numbers of the Ariki-Koike algebra $\He$, and 
show that similar results hold as in  the previous section.
\para
Let $\w=(-,\cdots,-,(1^n))$ be the $r$-partition and $T^\w$ be the $\w$-tableau of type $\w$. 
Since $\vf_{T^\w T^\w}$ is an identity map on $M^\w$ and a zero map $M^\mu$ for $\mu\in \vL$ such that $\mu \not= \w$, 
$\vf_{T^\w T^\w}$ is an idempotent in $\Sc$. 
Moreover we  see that $\vf_{T^\w T^\w}\Sc\vf_{T^\w T^\w}=\Hom_{\He}(M^\w,M^\w)=\Hom_{\He}(\He,\He)\cong \He$. 
It is well known that, for an $\Sc$-module $M$, $M\vf_{T^\w T^\w}$ becomes a $\He$-module 
	through the isomorphism $\vf_{T^\w T^\w}\Sc \vf_{T^\w T^\w}\cong \He$. 
Then we can define a functor, the so-called \lq\lq Schur functor", 
from the category of right $\Sc$-modules to the category of right $\He$-modules by $M \mapsto M \vf_{T^\w T^\w}$. 
The following facts are known by \cite[Proposition 2.17]{JM00}.
\begin{align}
\label{Weyl-Specht} W^\la \vf_{T^\w T^\w} \cong S^\la \quad \text{as }\He\text{-modules} \quad (\la\in \vL^+)\\
\label{L-D} L^\mu \vf_{T^\w T^\w} \cong D^\mu \quad \text{as }\He\text{-modules} \quad (\mu\in \vL^+)\\
[W^\la : L^\mu ]=[S^\la : D^\mu] \quad (\la,\mu \in \vL^+ \text{ such that }D^\mu \not=0)
\end{align}
where $[S^\la : D^\mu]$ is the decomposition number of $D^\mu$ in $S^\la$. 

\para
One can define the Jantzen filtration of the Specht module $S^\la$ in a similar way as in the case of $W^\la$, 
and we use a similar notation for this case. 
Then one can  define the $v$-decomposition number of $\He$, for $\la, \mu \in \vL^+$ such taht $D^\mu \not=0$, by
\[d^\He_{\la\mu}(v) :=\sum_{i\geq0} \big[S^\la(i)/S^\la(i+1) : D^\mu \big]\cdot v^i. \]
We have the following lemma.

\begin{lem} \label{W(i)-S(i)}
Let $\la \in \vL^+$ and $i\geq 0$. Under the isomorphism in (\ref{Weyl-Specht}), we have 
\[W^\la(i) \vf_{T^\w T^\w} =S^\la(i). \]
\end{lem}
\begin{proof}
It is clear that $W^\la\vf_{T^\w T^\w}$ has a basis $\{\vf_T\,|\, T\in \CT_0(\la,\w)\}$. 
We have a bijective correspondence between $\CT_0(\la,\w)$ and  $\Std(\la)$ by $T\leftrightarrow \ft$ such that $\w(\ft)=T$. 
Moreover, under the isomorphism in (\ref{Weyl-Specht}), we have 
\begin{align}\label{zzz}
&\vf_T\vf_{T^\w T^\w}=\begin{cases} m_{\ft} &\text{ if }T \in \CT_0(\la,\w)\\ 0 &\text{ if }T \not\in \CT_0(\la,\w) \end{cases}\\
\label{zzzzz}
&\lan \vf_S, \vf_T \ran=\lan m_{\fs}, m_{\ft} \ran_\He \quad \text{ for } S=\w(\fs),T=\w(\ft)\in \CT_0(\la ,\w )
\end{align}
by a similar argument as in the proof of \cite[Theorem 4.18]{M-book}.

First, we show the inclusion $W^\la(i)\vf_{T^\w T^\w} \subseteq S^\la(i)$. 
Take $x \in W_R^\la(i)$. Then $\lan x, \vf_T \ran \in \wp^i$ for any $T\in \CT_0(\la)$. 
It follows that
\[\lan x\cdot \vf_{T^\w T^\w}\,,\, \vf_T \ran =\lan x\,,\, \vf_T\cdot\vf_{T^\w T^\w}\ran \in \wp^i \quad\text{ for any } T\in \CT_0(\la).\]
This shows that 
\[\lan x \cdot \vf_{T^\w T^\w}\,,\, m_\ft \ran_\He \in \wp^i \quad \text{ for any }\ft \in \Std (\la)\]
by (\ref{zzz}) and (\ref{zzzzz}). Hence $x\cdot \vf_{T^\w T^\w} \in S_R^\la(i)$, and the claim follows by taking the quotient.

Next, we show the converse inclusion $W^\la(i)\vf_{T^\w T^\w} \supseteq S^\la(i)$. 
Take $y \in S^\la(i)$. Then we have 
\begin{equation}\label{zzxx}
\lan y, m_{\fs}\ran_\He \in \wp^i \quad\text{ for any }\fs\in \Std(\la).
\end{equation}
Write $y=\sum_{\ft\in \Std(\la)}r_{\ft}m_{\ft}$, and put $x=\sum_{T_\in \CT_0(\la,\w)}r_{\ft}\vf_T \in W^\la$, 
where $T$ is the $\la$-tableau of type $\w$ corresponding to $\ft$. Then we have $y=x\cdot \vf_{T^\w T^\w}$, and 
\[\lan x, \vf_S \ran \in \wp^i \quad \text{ for any }S\in \CT_0(\la,\w)\] 
by (\ref{zzz}), (\ref{zzzzz}) and (\ref{zzxx}). 
Since $\lan \vf_T, \vf_S \ran =0$ if the type of $T$ is not the same as the type of $S$, we have 
\[\lan x, \vf_S \ran \in \wp^i \quad \text{ for any } S\in \CT_0(\la).\]
This shows that $x\in W^\la_R(i)$, and the claim follows. The lemma is proved.
\end{proof}

This lemma implies the following proposition.
\begin{prop} \label{prop-W(i)-S(i)}
Take $\la,\mu \in \vL^+$ such that $D^\mu \not=0$. Then for any $i \geq 0$, we have 
\[\big[W^\la(i)/W^\la(i+1) : L^\mu \big]=\big[S^\la(i) /S^\la(i+1) : D^\mu \big]. \]
In particular, we have $d_{\la\mu}(v)=d_{\la\mu}^\He(v)$.
\end{prop}
\begin{proof}
We consider the $\Sc$-module filtration
\[W^\la(i)=W_0 \supsetneqq W_1 \supsetneqq \cdots \supsetneqq W_k=W^\la(i+1)\] 
such that $W_j/W_{j+1}\cong L^{\mu_j}$. 
By applying the Schur functor, together with Lemma \ref{W(i)-S(i)}, we have 
\[S^\la(i)=W^\la(i)\vf_{T^w T^w}\supset W_1\vf_{T^\w T^\w} \supset \cdots 
	\supset W_k\vf_{T^\w T^\w}=W^\la(i+1)\vf_{T^\w T^\w}=S^\la(i+1),\]
where $W_j\vf_{T^\w T^\w}/W_{j+1}\vf_{T^\w T^\w}\cong (W_j/W_{j+1})\vf_{T^\w T^\w}\cong L^{\mu_j}\vf_{T^\w T^\w} \cong D^{\mu_j}$ 
by (\ref{L-D}). The proposition follows from this.
\end{proof}

For $\la \in \vL^+$ such that $\alp(\la)=(n_1,\cdots,n_g)$, $\la^{[k]}$ is an $r_k$-partition of $n_k$. 
Then we have the Specht module $S^{\la^{[k]}}$ and its unique quotient $D^{\la^{[k]}}$ for the Ariki-Koike algebra $\He_{n_k,r_k}$. 
Moreover for $\la,\mu \in \vL^+$ such that $\alp(\la)=\alp(\mu)=(n_1,\cdots ,n_g)$, 
we have the $v$-decomposition number $d^\He_{\la^{[k]}\mu^{[k]}}(v)$ for $\He_{n_k r_k}$. 
Combining Theorem \ref{th-decom-bar-v} with Proposition \ref{prop-W(i)-S(i)}, we have the following result. 

\begin{thm} \label{th-vdecom-AK}
Let $\la,\mu \in \vL^+$ such that $\alp(\la)=\alp(\mu)$. Assume that $D^\mu \not= 0$ and $D^{\mu^{[k]}}\not=0 $ for any $k=1,\cdots ,g$. 
Then we have 
\[d_{\la\mu}^\He(v)=\prod_{k=1}^{g} d_{\la^{[k]}\mu^{[k]}}^\He(v).\] 
\end{thm}


\begin{thebibliography}{DJM10}

\bibitem[A1]{Ari96}
S.~Ariki.
\newblock On the decomposition numbers of the {H}ecke algebra of {$G(m,1,n)$}.
\newblock {\em J. Math. Kyoto Univ.} {\bf 36} (1996), 789-808.

\bibitem[A2]{Ari01}
S.~Ariki.
\newblock On the classification of simple modules for cyclotomic {H}ecke
  algebras of type {$G(m,1,n)$} and {K}leshchev multipartitions.
\newblock {\em Osaka J. Math.} {\bf38} (2001), 827-837.

\bibitem[AM]{AM00}
S.~Ariki and A.~Mathas.
\newblock The number of simple modules of the {H}ecke algebras of type
  {$G(r,1,n)$}.
\newblock {\em Math. Z.} {\bf233} (2000), 601-623.

\bibitem[DJM]{DJM98}
R.~Dipper, G.~James, and A.~Mathas.
\newblock Cyclotomic {$q$}-{S}chur algebras.
\newblock {\em Math. Z.} {\bf 229} (1998), 385-416.

\bibitem[DR]{DR98}
J.~Du and H.~Rui.
\newblock Based algebras and standard bases for quasi-hereditary algebras.
\newblock {\em Trans. Amer. Math. Soc.} {\bf350} (1998), 3207-3235.

\bibitem[GL]{GL96}
J.~J. Graham and G.~I. Lehrer.
\newblock Cellular algebras.
\newblock {\em Invent. Math.} {\bf123} (1996), 1-34.

\bibitem[Jac]{Jac05}
N.~Jacon.
\newblock An algorithm for the computation of the decomposition matrices for
  {A}riki-{K}oike algebras.
\newblock {\em J. Algebra} {\bf292} (2005), 100-109.

\bibitem[JM]{JM00}
G.~James and A.~Mathas.
\newblock The {J}antzen sum formula for cyclotomic {$q$}-{S}chur algebras.
\newblock {\em Trans. Amer. Math. Soc.} {\bf352} (2000), 5381-5404.

\bibitem[LLT]{LLT96}
A.~Lascoux, B.~Leclerc, and J.-Y. Thibon.
\newblock Hecke algebras at roots of unity and crystal bases of quantum affine
  algebras.
\newblock {\em Comm. Math. Phys.} {\bf181} (1996), 205-263.

\bibitem[LT]{LT96}
B.~Leclerc and J.-Y. Thibon.
\newblock Canonical bases of {$q$}-deformed {F}ock spaces.
\newblock {\em Internat. Math. Res. Notices}, (1996) 447-456.

\bibitem[M1]{Mat98}
A.~Mathas.
\newblock Simple modules of {A}riki-{K}oike algebras.
\newblock In \lq\lq {\em Group representations: cohomology, group actions and topology }",
 {\em Proc. Sympos. Pure Math.} vol.{\bf63},
 Amer. Math. Soc., 1998, pp.383-396.
\bibitem[M2]{M-book}
A.~Mathas.
\newblock {\em Iwahori-{H}ecke algebras and {S}chur algebras of the symmetric
  group}, {\em University Lecture Series} Vol.{\bf 15}, 
\newblock Amer. Math. Soc. 1999.

\bibitem[M3]{Mat04}
A.~Mathas.
\newblock The representation theory of the {A}riki-{K}oike and cyclotomic
  {$q$}-{S}chur algebras.
\newblock In \lq\lq {\em Representation theory of algebraic groups and quantum
  groups}", {\em Adv. Stud. Pure Math.} Vol. {\bf40}, Math. Soc.
  Japan, Tokyo 2004, pp. 261-320.

\bibitem[SW]{SW}
T.~Shoji and K.~Wada.
\newblock Cyclotomic $q$-Schur algebras associated to the Ariki-Koike algebra, preprint.

\bibitem[U]{Ugl00}
D.~Uglov.
\newblock Canonical bases of higher-level {$q$}-deformed {F}ock spaces and
  {K}azhdan-{L}usztig polynomials.
\newblock In \lq\lq {\em Physical combinatorics (Kyoto, 1999)}",  {\em
  Progr. Math.} vol. {\bf191}, Birkh\"auser Boston, Boston, 2000, pp. 249-299

\bibitem[VV]{VV99}
M.~Varagnolo and E.~Vasserot.
\newblock On the decomposition matrices of the quantized {S}chur algebra.
\newblock {\em Duke Math. J.}, {\bf100}, (1999), 267--297.

\bibitem[Y]{Yvo05}
X.~Yvonne.
\newblock A conjecture for {$q$}-decomposition matrices of cyclotomic
  {$v$}-{S}chur algebras.
\newblock {\em J. Algebra}, {\bf304}, (2006) 419--456.

\end{thebibliography}
\end{document}